\documentclass[12pt,oneside]{article}
\usepackage{amsmath,amstext,amssymb,amsthm,color}
\usepackage{hyperref}

\theoremstyle{plain}
\newtheorem{theorem}{Theorem} 
\newtheorem{lemma}[theorem]{Lemma}
\newtheorem{proposition}[theorem]{Proposition}
\newtheorem{corollary}[theorem]{Corollary}
 \theoremstyle{definition}

 \theoremstyle{remark}
\newtheorem*{remark*}{Remark} \newtheorem{remark}{Remark}

\newcommand{\R}{\mathbb{R}}
\newcommand{\Rd}{{\R^{d}}}
\newcommand{\Rdz}{{\R^{d}\setminus\{0\}}}
\newcommand{\N}{\mathbb{N}}
\newcommand{\Z}{\int^{\infty}_{0}}
\renewcommand{\leq}{\leqslant} \renewcommand{\le}{\leq}
\renewcommand{\geq}{\geqslant} \renewcommand{\ge}{\geq}
\def\({\left(}
\def\){\right)}
\def\[{\left[}
 \def\]{\right]}

\newcommand{\E}{\mathbb{E}}
\newcommand{\p}{\mathbb{P}}

\newcommand{\lC}{{\underline{c}}}
\newcommand{\uC}{{\overline{C}}}
\newcommand{\la}{{\underline{\alpha}}}
\newcommand{\ua}{{\overline{\alpha}}}
\newcommand{\lt}{{\underline{\theta}}}
\newcommand{\ut}{{\overline{\theta}}}

\definecolor{kb}{rgb}{0,0,0}
\definecolor{mr}{rgb}{0,0,0}
\definecolor{tg}{rgb}{0,0,0}
\newcommand{\kb}[1]{{\textcolor{kb}{#1}}}

\newcommand{\mr}[1]{{\textcolor{mr}{#1}}}

\newcommand{\tg}[1]{{\textcolor{tg}{#1}}}

\definecolor{ed}{rgb}{0,0,0}

\title{
Density and tails
of
unimodal
convolution semigroups
\thanks{The research was
supported in part by NCN grants 2011/03/B/ST1/00423 and 2012/07/B/ST1/03356.}
\thanks{Authors' affiliations and emails: Institute of Mathematics of the Polish Academy of Sciences, 
Institute of Mathematics and Computer Science,
  Wroc\l{}aw University of Technology, 
ul. Wyb. Wyspia\'{n}skiego
27, 50-370 Wroc\l{}aw, Poland, krzysztof.bogdan@pwr.wroc.pl, 
tomasz.grzywny@pwr,wroc.pl, michal.ryznar@pwr.wroc.pl.
}
}

\author{Krzysztof Bogdan \and
Tomasz Grzywny
\and Micha\l{} Ryznar
}
\begin{document}
\maketitle
\begin{abstract}
We give
sharp bounds for the
isotropic unimodal
probability convolution semigroups when their L\'evy-Khintchine
exponent has \kb{Matuszewska indices
strictly between $0$ and $2$. }
\end{abstract}
\section{Introduction}\label{introit}
Estimating Markovian semigroups is important from the point of view of theory and applications because they describe evolution phenomena and underwrite various forms of calculus.
Diffusion semigroups traditionally receive most attention \cite{MR2648271} but considerable progress has also been made
in studies of transition densities of
rather general jump-type Markov processes.
Such studies  are usually based on assumptions concerning the profile
of the jump or L\'evy kernel (measure) at the diagonal (origin) and at infinity \cite {MR2357678,
MR2806700}. As a rule the assumptions
can be viewed as approximate or weak scaling conditions for the L\'evy density, to which
some
structure conditions may be added, see \cite[(1.9)-(1.14) and Theorem 1.2]{ MR2357678}. Typical results consist of sharp two-sided estimates of the heat kernel for small and/or large times.

Transition semigroups of L\'evy processes allow for a
deeper insight and direct approach from several directions
thanks to their convolutional structure and the available Fourier techniques. For instance,
the upper bounds for
transition densities of isotropic L\'evy densities with relatively fast decay at infinity
are obtained in \cite{MR2827465,MR2886383}
by using Fourier inversion, complex integration, saddle point approximation or the Davies' method.
In this work we study one-dimensional distributions $p_t(dx)$ of
rather general isotropic unimodal L\'evy processes $X=(X_t,\,t\ge 0)$ in $\Rd$.
We focus on pure-jump isotropic unimodal L\'evy processes.
Thus, $X$ is a c\`adl\`ag stochastic process with distribution $\p$, such that $X(0)=0$ almost surely, the increments of $X$ are independent with rotationally invariant and radially nonincreasing density function $p_t(x)$ on $\Rdz$,  and the  following L\'evy-Khintchine formula holds
for $\xi\in\Rd$:
\begin{equation}\label{LKf}
{\E}
e^{i\left<\xi, X_t\right>}=\int_\Rd e^{i\left<\xi,
{x}\right>}p_t(dx)=e^{-t\psi(\xi)},\quad \text{ where }\quad
\psi(\xi)=\int_\Rd \(1- \cos \left<\xi,x\right>\) \nu(dx).
\end{equation}
Here and below $\nu$ is an isotropic unimodal {\it L\'evy measure} and $\E$ is the integration with respect to $\p$.  Further notions and definitions are given in Sections~\ref{sec:prel} and \ref{sec:HeatKernelRd} below.
Put differently, we study
the vaguely continuous spherically isotropic unimodal convolution semigroups $(p_t, t\ge 0)$ of probability measures on $\Rd$ with purely nonlocal generators.
(In
this work we never
use probabilistic techniques beyond the level of one-dimensional distributions of $X$.)

Our main result provides estimates for the tails of $p_t$ and its density function $p_t(x)$,
expressed in terms of the L\'evy-Khintchine exponent $\psi$.
We also
use $\psi$ to
estimate the density function $\nu(x)$ of the
L\'evy
measure $\nu$.
Since $\psi$ is radially {\it almost increasing}, it is {\it comparable}
with its radial nondecreasing majorant $\psi^*$, and as a rule we
employ $\psi^*$ in statements and proofs.
The extensive
useage of
$\psi$ ($\psi^*$) rather than $\nu$
is a characteristic feature of our development and
may be 
considered natural from the point of view of pseudo-differential operators
and spectral theory \cite{MR1873235}.
As usual for Fourier transform, the asymptotics of $\psi$ at infinity translates into the asymptotics of $p_t$ and $\nu$  at the origin. Our estimates may be summarized as follows,
\begin{equation}\label{allcomp}
p_t({x})\approx \[\psi^{-}\(1/t\)\]^{d}\wedge \frac {t\psi^*(|x|^{-1})}{|x|^d}\approx p_t(0)\wedge [t\nu({x})],
\end{equation}
see Theorem~\ref{densityApprox}, Corollary~\ref{cc} and \eqref{apt0} for detailed statements.
Here $\approx$ means that both sides are {\it comparable} i.e. their ratio is bounded between two positive constants,
$\psi$
is assumed to satisfy the so-called {\it weak upper} and {\it lower scalings}
of order strictly between $0$ and $2$,
and $\psi^-$ is the generalized inverse of {$\psi^*$}.
We shall see that \eqref{allcomp}
holds locally in time and space,
or even globally,
if the scalings are global.
We note that the corresponding
 estimates of
$\nu$, to wit,
\begin{align}\label{apnu}
\nu(x)&\approx \frac{\psi^*(|x|^{-1})}{|x|^d},
\end{align}
are simply obtained as a consequence of \eqref{allcomp}
and not as an element of its proof (see Corollary~\ref{cc}).
It is common for $\nu$ to share the asymptotics with $p_t$ because $\nu=\lim_{t\to 0} p_t\tg{/t}$, a vague limit in $\Rdz$.
It is also a manifestation of the general rule mentioned above that
$\psi^*(|x|^{-1})$ in our estimates reflect the properties of
$p_t(x)$ and $\nu(x)$.
The denominator, $|x|^d$, in the estimates comes from the homogeneity of the volume measure in $\Rd$ (see the proof of Corollary~\ref{cc}) and $\psi^-(1/t)$ approximates $p_t(0)$, as follows from a change of variables in Fourier inversion formula (cf. 
Lemma~\ref{inverse} and Lemma~\ref{ltst}).
All these reasons make the above
bounds quite natural. Therefore below we shall address (two-sided and one-sided) estimates similar to \eqref{allcomp} and \eqref{apnu} as {\it common bounds}.
We should
 note that the common upper bounds,
\begin{equation*}\label{onuzg1}
{p_t(x)\leq
C
 \frac {t\psi^*(1/|x|)} {|x|^d}},\quad\qquad\nu(x)\leq C \frac{\psi^*(1/|x|)}{|x|^{d}}, \qquad x\in{\Rdz},\; t>0,
\end{equation*}
hold with a constant depending only on the dimension for {\it all} isotropic unimodal L\'evy semigroups. This is proved in
Corollary~\ref{upper_den}.
The lower common bounds hold if and only if
$\psi$ has (the so-called weak) lower and upper scalings of order strictly between $0$ and $2$, \kb{equivalently if the lower and upper Matuszewska indices are strictly between $0$ and $2$.}
Namely,
in Theorem~\ref{NWSR} we show that for unimodal L\'evy processes,
the scaling of
$\psi$ (at infinity) is
{\it equivalent} to
the
common bounds
for the transition density and the L\'{e}vy measure
(at the origin).
In fact, already the
lower bound $\nu(x)\ge c \psi^*(|x|^{-1})/|x|^d$
implies such scalings of $\psi$.

We thus cover
all the cases of
isotropic L\'evy processes with \kb{scalings of order strictly between $0$ and $2$} studied in literature, and
upper bounds are provided for all  isotropic unimodal L\'evy processes.  
\kb{We leave
open the
problem of estimating
the semigroup of general unimodal L\'evy process with scaling
if  the upper Matuszewska index of $\psi$
is equal to $0$
or the lower Matuszewska index of $\psi$ is equal to $2$.
Here typical examples are the variance gamma process, i.e. the Brownian motion subordinated by an independent gamma subordinator, the geometric stable processes \cite{MR2569321}, and the conjugate geometric stable processes \cite{MR2886459}.
As we see from the results of \cite[Section~5.3.4]{MR2569321} and \cite{MR2886459},}
such processes require specialized approach, and their transition density and L\'evy-Khintchine exponent do not easily explain each other.

Bochner's procedure of subordination is strongly
rooted
in semigroup and operator theory,
harmonic analysis and probability \cite{MR1336382,MR0290095,MR2978140}, and keeps influencing
intense current developments. In the present setting it yields a
wide array of asymptotics of $\psi$, $p_t$ and $\nu$.
In particular,
common bounds
were
recently obtained for a class of subordinate Brownian motions, mainly for the {\it complete} subordinate Brownian motions
defined by a delicate structure condition \cite{MR2978140}. Highly sophisticated current techniques and results in this direction are presented in \cite{MR2986850}, see also \cite{MR2569321,MR2978140}.
Our approach
is, however,
more general
and synthetic; we
demonstrate
that  the
sources
of the asymptotics of (unimodal) $p_t(x)$
are merely
its radial monotonicity and
the scalings
of
$\psi$,
rather than
further
structure properties of $\psi$. \kb{Our arguments are also purely real-analytic.}
We illustrate our results with several classes of relevant examples. These include situations
where former methods cannot be easily applied and the present method works well.
To explain the current advantage,
we note that $\psi$ is an integral quantity and may exhibit less variability than $\nu$.

Many of our examples are subordinate Brownian motions, i.e. we have $\psi(\xi)=\phi(|\xi|^2)$, where $\phi$ is the Laplace exponent of the subordinator, or a Bernstein function.
There is by now
an impressive pool of Bernstein functions studied in the literature, with various asymptotics at infinity. For instance the monograph \cite{MR2978140}
gives well over one hundred cases and classes of
Bernstein functions in its closing list of examples. Many of these functions have the
scaling
properties used
in our paper, which
immediately yields sharp estimates of the L\'evy measure and transition density of subordinate Brownian motions corresponding to such subordinators.
\kb{In particular,
for subordinate Brownian motions with scaling, we
relax the
usual completeness assumption on
the subordinator.}
In comparison, former methods require to first find estimates of the L\'evy measure of the subordinator (this is where the completeness of the subordinator plays a role), then to estimate the L\'evy density of the corresponding subordinate Brownian motion and then, finally, to estimate its semigroup \cite{2013arXiv1303.6449C,MR2986850,MR2569321}.

We remark that analogues of the on-diagonal term $[\psi^-(1/t)]^d\approx p_t(0)$ in \eqref{allcomp}
are often
obtained for more general Markov processes
via Nash and Sobolev inequalities
\cite{MR2299447,MR898496,MR2492992,MR1218884,MR2886383}. For our unimodal L\'evy processes we
instead use Fourier inversion and (weak) lower scaling, see Proposition~\ref{sup_p_t_psi}.
Also our approach to the off-diagonal term ${t\psi^*(|x|^{-1})}/|x|^d$ is very different and much simpler than the arguments leading to the upper bounds in the otherwise more general Davies' method \cite{MR898496}, \cite[Section~3]{MR990239}.
Our common upper bounds are straightforward consequences of a specific quadratic parametrization of the tail function, which is crucial in
applying the techniques of
Laplace transform. The common lower bounds are
harder, and they are intrinsically
related to upper and lower scalings via certain differential inequalities in the proof of Theorem~\ref{NWSR}.

The (local or global) comparability of the common lower and upper bounds is a remarkable feature of
the class of semigroups captured by Theorem~\ref{NWSR}.
We expect
further applications of the estimates.
For instance,
under (weak) global scalings
we obtain important metric-type \cite{MR2925579} global comparisons $p_{2t}(x)\approx p_t(x)$ and $p_t(2x)\approx p_t(x)$, given in Corollary~\ref{Doubling} below. These should matter in perturbation theory of L\'evy generators and in nonlinear partial integro-differential equations. Since uniform
estimates are important in some applications, 
the {\it comparability constants} in the paper are generally shown to depend in a rather explicit way on specific properties of the semigroups, chiefly on scaling. 
Noteworthy, our (weak) scaling conditions
imply
majorization and minorization of $\psi$ at infinity by power functions with exponents strictly between $0$ and $2$, but do not require its {\it comparability}
with a power function (see examples in Section~\ref{sex}). Furthermore, the exponents $\la$ and $\ua$ in the assumed scalings
only affect the comparison constants in the common bounds, but not the rate of asymptotic, which is solely determined by $\psi$.

For convolution semigroups of probability measures more general than unimodal, the structure of the support and the spherical regularity of the L\'evy measure
plays a crucial role.
In particular the directions which are not charged by the L\'evy measure see in general lighter asymptotics of $p_t(dx)$ \cite{MR1085177, MR2286060}.
In consequence, the estimates of severely anisotropic convolution semigroups require completely different assumptions, description and methods. Our experience indicates that $\nu$ surpasses $\psi$ in such cases. Estimates and references to anisotropic
$\nu$ with prescribed radial decay and rough spherical marginals may be found in
\cite{MR2794975} (see also \cite{MR2320691, MR2286060} for more details in the case of {\it homogeneous} anisotropic $\nu$).

The structure of the paper is as follows. In Section~\ref{sec:prel} we discuss first consequences of isotropy and radial monotonicity. In particular we compare $\psi$ with Pruitt-type function $h_1(r)=\int_{\Rd}\(1\wedge |x_1|^2/r^2\)\nu(dx)$ and we estimate from above
the tail function ${f}_t(\rho)=\p(|X_t|\ge \sqrt{\rho})$ by
using Laplace transform and $\psi$. This and radial monotonicity quickly lead to upper bounds for $p_t(x)$ and $\nu(x)$.
In Section~\ref{sec:HeatKernelRd} we discuss almost monotone and general weakly scaling functions.
In Section~\ref{sec:SLKe} we specialize to scalings
with lower and upper exponents $\la,\ua\in (0,2)$,
and we give examples of $\psi$ with such scaling. \kb{We also explain the relationship to Matuszewska indices.} To obtain lower bounds for $p_t(x)$ and $\nu(x)$,
in Lemma~\ref{LaplaceLower} we recall an observation due to M. Z\"ahle, which is then used in Lemma~\ref{tail} to reverse the comparison between
$\psi$
 and the tail function $f_t$.
The generalized inverse $\psi^-$ plays a role through a change of variables in Fourier inversion formula for $p_t(0)$ in Lemma~\ref{inverse} and through
equivalence relation defining "small times" (stated as Lemma~\ref{ltst}). In Theorem~\ref{densityApprox} we combine all the threads to estimate $p_t$, as summarized in \eqref{allcomp}. In Corollary~\ref{cc} we obtain \eqref{apnu}
as a simple consequence of Theorem~\ref{densityApprox}. To close the circle of ideas, in Theorem~\ref{NWSR} we show the equivalence of the weak scalings with the
common bounds for
$p_t$ and $\nu$.
In Proposition~\ref{tau} we state a
connection between $\nu$ and $\psi$ for a class of {\it approximately} isotropic L\'evy densities.

\section{Unimodality}\label{sec:prel}
We
shall often use the
gamma and incomplete gamma functions:
$$\Gamma(\delta)=\int^\infty_0 e^{-u}u^{\delta-1}du, \qquad \gamma(\delta,t)=\int^t_0e^{-u}u^{\delta-1}du,\qquad \Gamma(\delta,t)=\int^\infty_te^{-u}u^{\delta-1}du,\qquad \delta,\,t>0.$$
Let $\R^d$ be the Euclidean space of (arbitrary) dimension $d\in \N$.
For $x\in \Rd$ and $r>0$ we let $B(x,r)=\{y\in \Rd: |x-y|<r\}$,
and
$B_r=B(0,r)$. We denote by 
\mr{$\omega_d=2\pi^{d/2}/\Gamma(d/2)$} 
the surface measure of the unit sphere in $\Rd$.
All sets, functions and measures considered below are
(assumed)
Borel.
A
(Borel) measure
on $\Rd$ is called isotropic unimodal (in short, unimodal) if
on $\Rd\setminus \{0\}$ it is absolutely continuous with respect to the Lebesgue measure,
and has a finite {\it radial nonincreasing} density function. Such measures may have
an atom at the origin: they are of the form $
a\delta_0(dx)+f(x)dx$, where $a\ge 0$, $\delta_0$ is the Dirac measure,
$$
f(x)=\int_0^\infty {\bf 1}_{B_r}(x)\mu(dr)=\mu\big((|x|,\infty)\big) \qquad (a.e.),
$$
\mr{and $\mu$ is a 
measure on $\kb{(}0, \infty)$} such that  $\mu\big((\varepsilon,\infty)\big)<\infty$ for all $\varepsilon>0$.
A L\'evy process $X=(X_t, \,t\ge 0)$, is called isotropic unimodal (in short, unimodal) if all of its one-dimensional distributions (transition densities) $p_t(dx)$ are such.
Recall that L\'evy measure is
any
measure
concentrated on $\Rdz$ such that
\begin{equation}\label{wml}
\int_\Rd \(|x|^2\wedge 1\)\nu(dx)<\infty.
\end{equation}
Unimodal pure-jump L\'evy processes are characterized in \cite{MR705619}
by
unimodal L\'evy measures
$\nu(dx)=\nu(x)dx=\nu(|x|)dx$.

Unless explicitly stated otherwise, in what follows we
assume that
$X$ is a pure-jump
{unimodal} L\'{e}vy process in $\Rd$  with (unimodal) {\it nonzero} L\'evy measure (density) $\nu$.

Each measure $p_t$ is the weak limit of
$$p_t^{\varepsilon}=e^{-t\nu_\epsilon(\Rd)}\sum_{n=0}^\infty \frac{1}{n!}\nu_\epsilon^{{*}n},$$
where $\epsilon \to 0^+$ and $\nu_\epsilon^{{*}n}$ are convolution powers of measures $\nu_\epsilon(A)=\nu(A\setminus B_\epsilon)$,
see, e.g., \cite[Section 1.1.2]{MR2569321}.
Each $p_t$
has a radial nonincreasing density function $p_t(x)$ on $\Rdz$
and atom
$\exp[-t\nu(
\Rd)]$ at the origin if $\nu(
\Rd)<\infty$ (no atom if $t\nu$ is infinite).

For $r>0$ we define after W.~Pruitt \cite{MR632968},
\begin{equation}\label{def:GKh}
h(r)=\int\limits_{\Rd}\(\frac{|{x}|^2}{r^2}\wedge 1\)\tg{\nu(dx)},\qquad\qquad L(r)= \nu\(B^c_r\).\qquad
\end{equation}
Clearly, $0\leq L(r)< h(r)<\infty$,
$L$ is nonincreasing and $h$ is  decreasing. The strict monotonicity and positivity of $h$ follows since $\nu\neq 0$ is nonincreasing, hence positive near the origin.

The
first coordinate process
$X_t^{1}$ of
$X_t$
is
unimodal in $\R$. The corresponding quantities $L_1(r)$ and $h_1(r)$ are given by
the (pushforward) L\'evy  measure $\nu_1=\nu\circ {\pi}_1^{-1}$,
where $\pi_1$ is the projection: ${\Rd\ni} x=(x_1,\ldots,x_d)\mapsto x_1$, see  \cite[Proposition~11.10]{MR1739520}.
With a typical abuse of notation we let $\nu_1(y)$ denote
the (symmetric and nonincreasing on $(0,\infty)$) density function of $\nu_1$:
$$
\nu_1(y)=\int_{\R^{d-1}}\nu(\sqrt{y^2+|z|^2})dz,\qquad y\in \R\setminus\{0\}.
$$
Thus,
$$h_1(r)=\int_\R \(\frac{
y^2}{r^2}\wedge 1\)\nu_1(dy)=
\int\limits_{\Rd}\(\frac{|
{x}_1|^2}{r^2}\wedge 1\)\nu(d
x),\quad r>0.
$$
Therefore,
\begin{equation}\label{coh}
h_1(r)\le h(r)\le h_1(r)d,\quad r>0.
\end{equation}
In fact, \eqref{coh} is valid  for
all rotation invariant L\'evy measures.
\begin{remark}\label{sch}
\mr{The functions} $h(r)$ and $h_1(r)$ are decreasing, while $r^2h(r)$ and $r^2 h_1(r)$ are nondecreasing.
\end{remark}

Since
$\psi$ is a radial function,  we shall often
write
$\psi(u)=\psi(
\xi)$, where $\xi\in \Rd$ and $u=|\xi|\ge 0$.
We obtain the same function for
$X_t^1$.
Clearly, $\psi(0)=0$ and, as before for $h$, $\psi(u)>0$ for $u>0$.

We now show how to use
$h_1$ to estimate the L\'evy-Khintchine exponent $\psi$ of $X$.
\begin{lemma}\label{l:hop}
\begin{equation}\label{eqckh1}
\frac{2}{\pi^2}h_1\(\frac1u\)\le \psi(u)\le 2 h_1\(\frac1u\),\quad u>0.
\end{equation}
\end{lemma}
\begin{proof}
For $t\ge 0$ we define $\kappa(t)=\int_0^t(1-\cos r)dr$, and we claim that
\begin{equation}\label{eqoka}
\frac{2}{\pi^2}\int_0^t (r^2\wedge 1) dr\le \kappa(t) \le 2\int_0^t (r^2 \wedge 1) dr,\quad t\ge 0.
\end{equation}
Indeed, $1-\cos r=2\sin^2 (r/2)\le
2(r^2\wedge 1)$, which gives the upper bound. If $0\le r\le t\le \pi$, then
$1-\cos r=2\sin^2 (r/2)
\ge 2r^2/\pi^2$ hence, $\kappa(t)\ge 2\pi^{-2}\int_0^t (r^2\wedge 1) dr$, and if $t>\pi$, then
\begin{align*}
\int_{\pi/2}^t (1-\cos r)dr&=(t-\pi/2)+(1-\sin t)\ge
\int_{\pi/2}^t dr,
\end{align*}
which yields \eqref{eqoka}. (The constants in \eqref{eqoka} may be improved.)
We define an auxiliary measure $\mu$ on $(0,\infty)$ by letting
$$\mu((y,\infty))=\nu_1(y)\quad  \mbox{for a.e. $y>0$.}$$
Let $u
>0$.
By a change of variables and Fubini-Tonelli,
\begin{align*}
\frac12h_1\(\frac1u\)&=\int_0^\infty [(u^2y^2)\wedge 1]\nu_1(y)dy
=\int_0^\infty [(u^2y^2)\wedge 1]\int_{(y,\infty)}\mu(dt)dy\\
&=\int_0^\infty \int_0^t[(u^2y^2)\wedge 1]dy\,\mu(dt)=
\int_0^\infty \frac1u\int_0^{tu}[r^2\wedge 1]dr\,\mu(dt).
\end{align*}
Similarly,
\begin{align*}
{\frac12}\psi(u)&=
\int_0^\infty \[1-\cos (uy)\]\nu_1(y)dy=
\int_0^\infty \[1-\cos (uy)\]\int_{(y,\infty)}\mu(dt)dy=
\int_0^\infty
\frac1u\kappa(tu)\mu(dt).
\end{align*}
By these identities and \eqref{eqoka} we obtain \eqref{eqckh1}.
\end{proof}

We define the maximal characteristic function $\psi^*(u):= {\sup_{s\leq {u}}\psi(s)}$, where $u\ge 0$.
The following result is a version of \cite[Proposition 1]{Grzywny}.
\begin{proposition} \label{Vestimate}
\begin{equation}\label{eqcpg}
\psi(u)\le \psi^*(u)\leq \pi^2\, \psi(u) \quad \text{ for } \quad u\ge0.
\end{equation}
\end{proposition}
\begin{proof}
Since $h_1$ is nonincreasing, by Lemma~\ref{l:hop} for $u\ge0$ we have
\begin{align*}
\psi(u)\le \psi^*(u)\le 2\sup_{0<s\le u}h_1\(\frac1s\)=2h_1\(\frac1u\)\le\pi^2\psi(u).
\end{align*}
\end{proof}
We write $f(x)\approx g(x)$ and say $f$ and $g$ are comparable if $f, g\ge 0$ and there is a positive number $C$, called {\it comparability constant},
such that $C^{-1}f(x)\le g(x)\le C f(x)$ for all $x$.
We write $C=C(a,\ldots,z)$ to indicate that $C$
may be so chosen to depend only on $a,\ldots,z$.
We say the comparison is absolute if the constant is absolute.
Noteworthy,
while $\psi$ is comparable to a nondecreasing function, it need not be nondecreasing itself.
For instance, if $\psi(
{u}) = u + 3\pi
[1-(\sin u)/u]$, then $
{\psi}'(2\pi)=-{\frac12}
<0$.

The following conclusion
may be interpreted as
relation of ``scale'' and ``frequency''.
\begin{corollary}\label{ch1Vpc}
We have ${h(r)\approx}h_1(r)
\approx
\psi(1/r)\approx
\psi^*(1/r)$ for
$r>0$.
\end{corollary}
\begin{proof}
The constant in the
{leftmost} comparison depends only on the dimension,
see \eqref{coh}.
The other comparisons are absolute,
by Lemma~\ref{l:hop} and Proposition~\ref{Vestimate}.
\end{proof}
By Corollary \ref{ch1Vpc} and definitions of $L_1$, $L$ and $h$, we obtain the following inequality:
\begin{equation}\label{tails}
L_1(r)\le
L(r)
< h(r)\le
{C}\psi^*(1/r), \qquad r>0,
\end{equation}
\kb{where $C=C(d)$.}
Our main goal is to
describe asymptotics of $\nu(x)$ and $p_t(x)$ in terms of $\psi^*$.
We start with the
analysis of the Laplace transform of the integral tails of $p_t$.
For reasons which shall become clear in the proof of the next result, we choose the following parametrization of the tails:
\begin{align}
{f}_t(\rho)&=\p(|X_t|\ge \sqrt{\rho})=\p(|X_t|^2>\rho)\label{pate}
,\qquad\rho\ge 0, \quad t>0.
\end{align}
Consider the Laplace transform of $f_t$:
$${\cal L}{{f}_t}(\lambda)=\int_0^\infty e^{-\lambda\rho}f_t(\rho)d\rho,\qquad \lambda\ge 0.$$

\begin{lemma}\label{laplace}
There is a constant $C_1=C_1(d)$ such that
$$C_1^{-1}   \frac 1\lambda\(1-e^{-t\psi^*(\sqrt{\lambda})}\) \le {\cal L}{{f}_t}(\lambda)\le C_1\frac 1\lambda \(1-e^{-t\psi^*(\sqrt{\lambda})}\) , \qquad \lambda >0.$$
\end{lemma}
\begin{proof}
By Fubini's theorem, $\int_\Rd \hat{h}(x)k(x)dx=\int_\Rd h(x)\hat{k}(x)dx$ for integrable functions $h$, $k$. By this,
\eqref{pate} and change of variables we obtain
\begin{align*}
\lambda{\cal L}f_t(\lambda)&=\E (1-e^{-\lambda|X_t|^2})\\
&=1-\int_\Rd e^{-\lambda|x|^2}p_t(x)dx
=(4\pi)^{-d/2}\int_{\Rd} \(1- e^{-t\psi\( x \sqrt{\lambda}   \)}\)e^{-|x|^2/4}dx.
\end{align*}
By \cite[Theorem~2.7]{1998-WHoh-habilitation},
\begin{equation}\label{Psi*ScalingGeneral}
\psi(s
{u})\leq\psi^*(s
u)\le 2
(s^2+1)\psi^*(
u
),
\qquad
s,u
\ge 0.
\end{equation}
(The estimate may usually be improved for specific $\psi$.)
We also note that
\begin{equation}\label{nfw}
1-e^{-bt}\le b(1-e^{-t}),\qquad
t\ge 0, \quad
b\ge 1,
\end{equation}
and we obtain
\begin{align}
\lambda{\cal L}f_t(\lambda)&\le  (4\pi)^{-d/2}\int_{\Rd} \(1- e^{-2t (|x|^2+1)\psi^*(\sqrt{\lambda} )}\)e^{-|x|^2/4}dx\nonumber\\
&\le 2 (2d+1)\(1- {e^{-t\psi^*(\sqrt{\lambda} )}}\) .\label{ubLt}
\end{align}
On the other hand, if
$|x|\geq1$, then
$\psi\(
 x
\sqrt{\lambda}\)\ge
\psi^*\(|x|\sqrt{\lambda}\){/\pi^2}
\ge
\psi^*\(\sqrt{\lambda}\)/\pi^2
$
by \eqref{eqcpg}.
Thus,
$$\lambda{\cal L}f_t(\lambda)\geq (4\pi)^{-d/2}\int_{B^c_1} e^{-|x|^2/4}dx\(1- {e^{-t\psi^*\(\sqrt{\lambda} \)/\pi^2}}\)\ge
\frac{\Gamma(d/2,1/4)}{\pi^2\Gamma(d/2)}
\(1- {e^{-t\psi^*\(\sqrt{\lambda} \)}}\),$$
where we use \eqref{nfw} and the upper incomplete gamma function $\Gamma(\cdot,\cdot)$.
\end{proof}
The upper bounds for tails
shall follow from this
auxiliary lemma.
\begin{lemma}\label{LaplaceUpper} If $f$ is nonnegative and nonincreasing, then for $n,m=0,1,2,\ldots$ and $
{r}>0$,
\begin{equation*}f(r)\leq \frac1{\gamma(n+m+1,1)} r^{-n-m-1}|(\mathcal{L}[s^mf])^{(n)}(r^{-1})|.
\end{equation*}
\end{lemma}
\begin{proof}If $
{u}>0$ and $r=u^{-1}$, then
\begin{align*}
u^{n+m+1}|(\mathcal{L}[s^mf])^{(n)}(u)|&=u^{n+m+1}|(-1)^n\mathcal{L}[s^{n+m}f](u)|\\&
\ge u^{n+m+1}\int^{u^{-1}}_0e^{-su}s^{n+m}f(s)ds
\geq f(u^{-1})\int^{u^{-1}}_0 u(us)^{n+m}e^{-su}ds\\
&=f(u^{-1})\int^{1}_0u^{n+m}e^{-u}du=f(u^{-1})\gamma(n+m+1,1),
\end{align*}
where we use the lower incomplete gamma function $\gamma(\cdot,\cdot)$.
\end{proof}
The following estimate
results from \eqref{ubLt} and  Lemma \ref{LaplaceUpper} with $n=m=0$.
\begin{corollary}\label{ubt}
For $
r>0$ we have
$ \p(|X_t|\ge r)\le
\frac{2e}{e-1}(2d+1)
\(1-e^{-t\psi^*(1/r )}\)$.
\end{corollary}
Here is a general upper bound for
density of unimodal L\'evy process.
(As we shall see in Theorem~\ref{densityApprox} and \ref{NWSR}, a reverse inequality often holds, too.)
\begin{corollary}\label{upper_den}
There is
$C=C(d)$ such that
$p_t(x)\leq C t\psi^*(1/|x|)/ |x|^{d}$ for $x\in{\Rdz}$.
\end{corollary}
\begin{proof}By radial monotonicity of $y\mapsto p_t(y)$, \eqref{Psi*ScalingGeneral} and Corollary~\ref{ubt},
$${p_t(x)\leq
\frac{\p \(  |x|/2\leq |X_t|<|x|\)}{ \left|B_{|x|}\setminus B_{|x|/2}\right|}
\leq   \frac{d}{(1-2^{-d})\omega_d}\p \(   |X_t|\ge \frac{|x|}{2}\)|x|^{-d} \le
C
 \frac {t\psi^*(1/|x|)} {|x|^d}.}$$
\end{proof}
Since $p_t(x)/t\to \nu(x)$ vaguely on $\Rdz$, we also obtain
\begin{equation}\label{onuzg}
\nu(x)\leq C \frac{\psi^*(1/|x|)}{|x|^{d}}, \qquad x\in{\Rdz}.
\end{equation}
Tracking constants, e.g., for the isotropic $\alpha$-stable L\'evy process,
we get
\begin{equation}\label{defCa1}
p_t(x)\leq \frac{ d4^\alpha\Gamma\(\frac{d+\alpha}{2}\)}{2(1-2^{-d})\pi^{d/2} \gamma\(1,1\)}
\;\frac{t}{|x|^{d+\alpha}},\qquad t>0, x\in \Rd.
\end{equation}
(In fact, for this constant we override \eqref{Psi*ScalingGeneral} by $\psi(su)=s^\alpha\psi(u)$ in the proof of Lemma~\ref{laplace}.)

\section{
Weak scaling and monotonicity
}\label{sec:HeatKernelRd}

\kb{Scaling conditions became standard in estimates of heat kernels \cite{MR2986850}.
For reader's convenience we give a short survey of weakly scaling and almost monotone functions,
because
some of the quantitative results given below 
are
difficult to find in references, cf.  \cite{MR898871,MR2986850}.}

Let $\phi:I\to [0,\infty]$, for a connected set $I\subset [-\infty,\infty]$.
First, we call $\phi$ {\it almost increasing} if there is (oscillation factor) $c\in (0,1]$ such that $c\phi(x)\le \phi(y)$ for  $x,y\in I$, $x\le y$.
Let
$$
\phi^*(y)=\sup\{\phi(x):\; x\in I,\, x\le y\}, \qquad y\in I.
$$
We easily check that $\phi^*$ is nondecreasing, $\phi\le \phi^*$ and the following result holds.
\begin{lemma}\label{cai}
$\phi$ is almost increasing with oscillation factor $c$ if and only if $c\phi^*\le \phi$.
\end{lemma}
E.g. the L\'evy-Khintchine exponent $\psi$ in Proposition~\ref{Vestimate} is almost increasing with factor $1/{{\pi^2}}$.
On the other hand, if there is $C\in [1,\infty)$ such that $C\phi(x)\ge \phi(y)$ for $x,y\in I$, $x\le y$, then we call $\phi$ {\it almost decreasing} (with oscillation factor $C$).
Let
$$
\phi_*(x)=\sup\{\phi(y):\; y\in I,\, y\ge x\}, \qquad x\in I.
$$
We easily check that $\phi_*$ is nonincreasing, $\phi\le \phi_*$ and the following result holds.
\begin{lemma}
$\phi$ is almost decreasing with oscillation factor $C$ if and only if $ \phi_*\le C\phi$.
\end{lemma}
We note that $\phi$ is almost increasing on $I$ with factor $c$ if and only if $1/\psi$ is almost decreasing on $I$ with factor $1/c$.
Here is another simple observation which we give without proof.
\begin{lemma}\label{connect}
Assume that
sets $I_1,I_2, I=I_1\cup I_2\subset \R$ are connected.
If $\phi$ is almost increasing (decreasing) on $I_1$ with factor ${c'}
$ ($
{C'}$), almost increasing (decreasing) on $I_2$ with factor ${c''}
$ ($C''
$), then $\phi$ is almost increasing (decreasing) on $I$ with factor $c=c'
c''
$ ($C=C
'
C''
$).
\end{lemma}
We say that
$\phi$ satisfies {the} weak lower scaling condition (at infinity) if there are numbers
$\la\in \R$, $\lt\ge 0$,
and  $\lC
\in(0,1]$,  such that
\begin{equation}\label{eq:LSC}
 \phi(\lambda\theta)\ge
\lC\lambda^{\,\la} \phi(\theta)\quad \mbox{for}\quad \lambda\ge 1, \quad\theta
\kb{>\lt}.
\end{equation}
In short we say that $\phi$ satisfies WLSC($\la, \lt,\lC$) and write $\phi\in$WLSC($\la, \lt,\lC$).
If $\phi\in$
WLSC($\la,0,\lC$), then we say
that $\phi$ satisfies the {\it global} weak lower scaling condition.

Similarly, 
the weak upper scaling condition holds if there are numbers $\ua\in \R$, $\kb{\ut\ge 0}$
and $\uC\in [1,\infty)$ such that
\begin{equation}\label{eq:USC}
 \phi(\lambda\theta)\le
\uC\lambda^{\,\ua} \phi(\theta)\quad \mbox{for}\quad \lambda\ge 1, \quad\theta
\kb{>\ut.}
\end{equation}
In short, $\phi\in$
WUSC($\ua, \ut,\uC$). For {\it global} weak upper scaling we require $\ut=0$ in \eqref{eq:USC}.

\kb{Here is a characterization of the scaling conditions in terms of almost monotone functions.}
\begin{lemma}\label{sam}
We have $\phi\in$WLSC($\la$,$\lt$,$\lC$) if and only if $\phi(\theta)= \kappa(\theta)\theta^\la$
and $\kappa$ is almost increasing on $(\lt,\infty)$ with oscillation factor $\lC$.
Similarly, $\phi\in$WUSC($\ua$,$\ut$,$\uC$) if and only if $\phi(\theta)= \kappa(\theta)\theta^\ua$
and $\kappa$ is almost decreasing on $(\ut,\infty)$ with oscillation factor $\uC$.
\end{lemma}
\begin{proof}
Let
$\lt\in [0,\infty)$, $I=(\lt,\infty)$ and $\phi\in$WLSC($\la$,$\lt$,$\lC$).
Let
$\kappa(\theta)=\phi(\theta)\theta^{-\la}$
on $I$.
If $\lt
< \eta\le \theta$
and $\lambda=\theta/\eta$, then
$$\kappa(\theta)=\phi(\lambda\eta)(\lambda\eta)^{-\la}\ge \lC \lambda^{\,\la}\phi(\eta)(\lambda\eta)^{-\la}=\lC\kappa(\eta).$$ Thus, $\kappa$ is almost increasing.
On the other hand, if $\phi(\theta)= \kappa(\theta)\theta^\la$ and $\kappa$ is almost increasing with factor $\lC$, then for $\lambda\ge 1$ and $\theta\in I$ we have
$$\phi(\lambda\theta)= \kappa(\lambda\theta)(\lambda\theta)^\la
\ge\lC\kappa(\theta) (\lambda\theta)^\la =\lC\lambda^\la\phi(\theta).$$
Thus, $\phi\in$WLSC($\la$,$\lt$,$\lC$).
The proof of the second part of the statement is left to the reader.
\end{proof}

\kb{
We make a connection of the scaling conditions to Matuszewska indices.}
\begin{remark}\label{rem:Mi}
\kb{Let $\phi:[0,\infty)\to [0,\infty)$ and let $\beta(\phi)\le \alpha(\phi)$ be the  lower and upper Matuszewska indices \cite[p. 68]{MR898871} of $\phi$, respectively. By \cite[Theorem~2.2.2]{MR898871},
if $\phi\in$WLSC($\la, \lt,\lC$) for some $\lt\ge 0$ and $\lC\in (0,1]$, then $\beta(\phi)\ge \la$, and if
$\phi\in$WUSC($\ua, \ut,\uC$) for some $\ut\ge 0$ and $\uC\in [1,\infty)$, then $\alpha(\phi)\le \ua$.
As a partial converse we have that, if $\la<\beta(\phi)$, then  $\lt\ge 0$ and $\lC\in (0,1]$ exist such that  $\phi\in$WLSC($\la, \lt,\lC$),
and if $\ua>\alpha(\phi)$, then $\ut\ge 0$ and $\uC\in [1,\infty)$ exist such that $\phi\in$WUSC($\ua, \ut,\uC$). 
We note that the scalings may, but need not hold for $\la=\beta(\phi)$ and $\ua=\alpha(\phi)$.
Furthermore, 
in what follows it is important 
to specify the ranges of $\theta$ for which the inequalities in \eqref{eq:USC} and \eqref{eq:LSC} hold, in particular, the  cases $\lt=0$ and $\ut=0$ are qualitatively different from the cases $\lt>0$ and $\ut>0$. 
These remarks explain why we need to state our assumptions in terms of weak scaling, rather than only use 
Matuszewska indices.}
\end{remark}

By Lemma \ref{sam}, WLSC($0, \lt, \lC$) characterizes almost increasing functions on $(\lt,\infty)$, and WUSC($0, \ut, \uC$) characterizes almost decreasing functions on $(\ut, \infty)$.

For example, $h,h_1\in$ WUSC($0$,$0$,$1$)$\cap$WLSC($-2$,$0$,$1$), see Remark~\ref{sch}.
\begin{remark}\label{SvC}
If $\phi\in$WLSC($\la$,$\lt$,$\lC$) and $c \phi\le \varphi\le C\phi$, then
$\varphi\in$WLSC($\la$,$\lt$,$\lC \,c/C$). Similarly,
if $\phi\in$WUSC($\ua$,$\ut$,$\uC$) and $c \phi\le \varphi\le C\phi$, then
$\varphi\in$\kb{WUSC}($\ua$,$\ut$,$\uC\, C/c$).
\end{remark}

As  $\lt$ or $\ut$ decrease, the scaling conditions tighten. Here is a loosening observation.
\begin{lemma}\label{scaling} Let
$\phi:[
{\theta_1},\infty)\to
(0,\infty)$ be \kb{nondecreasing.}
If $ 0<\theta_1< \lt$ and $\phi\in$
WLSC($\la, \lt, \lC$), then
${\phi\in}$WLSC($\la, \theta_1, {c}_1$) with
{ $c_1= \lC
\(\theta_1/\lt\)^{|\la|}\phi(\theta_1)/\phi(\lt)$.
}
If rather $0<\theta_1< \ut$ and $\phi\in$
WUSC($\ua, \ut, \uC$), then
${\phi\in}$WUSC($\ua, \theta_1, {C}_1$) with
 $C_1= \uC
\(\ut/\theta_1\)^{|\ua|}\phi(\ut)/\phi(\theta_1)$.
\end{lemma}
\begin{proof}
In view of Lemma~\ref{connect} and Lemma~\ref{sam}, it is enough to
study $\kappa(\theta)=\phi(\theta)\theta^{-\la}$
on $[\theta_1,\lt]$.
If $\theta_1\le \theta \le \lt$, then
$\phi(\theta_1)\(\lt^{-\la}\wedge \theta_1^{-\la}\)\le \kappa(\theta)\le
\phi(\lt)\(\theta_1^{-\la}\vee \lt^{-\la}\)$.
This gives the first implication and the second obtains similarly.
\end{proof}
\begin{remark}\label{inverseScal}
Let $\phi\ge 0$ be continuous and increase to infinity. If ${\phi}\in \textrm{WLSC}(\la,\lt,\lC)$ [$\textrm{WUSC}(\ua,\ut,\uC)$],
then $\phi^{-1}\in\textrm{WUSC}(1/\la,\phi(\lt),\lC^{-1/\la})$ [$\textrm{WLSC}(1/\ua,\phi(\ut),\uC^{-1/\ua})$, resp.].
 Indeed, since $\phi$ is increasing, by the scaling for $\lambda\geq 1$ and $\theta>\lt$ we have
$$\lambda\phi^{-1}(\phi(\theta))=\lambda \theta =\phi^{-1}(\phi(\lambda \theta))\geq \phi^{-1}(\lC \lambda^\la\phi(\theta)).$$
Thus for (arbitrary) $y=\phi(\theta)>
\phi(\lt)
$, {if}
$s=\lC\lambda^\la\geq \lC$, in particular if $s\ge 1$, then
$$ \lC^{-1/\la}s^{1/\la}\phi^{-1}(y)\geq \phi^{-1}(sy).$$
Similarly, if $\phi\in \textrm{WUSC}(\ua,\ut,\uC)$, then
 $$\uC^{-1/\ua}s^{1/\ua}\phi^{-1}(y)\leq\phi^{-1}(sy), \quad y>\phi(\ut),\, s\geq \uC.$$
For $1\leq s<\uC$, by monotonicity of $\phi^{-1}$,
$$\uC^{-1/\ua}s^{1/\ua}\phi^{-1}(y)\leq \phi^{-1}(y)\leq \phi^{-1}(sy),\quad y>\phi(\ut).$$
This proves our claim.
\end{remark}
\begin{remark}\label{invitp}
We also note that $\phi\in$WLSC($\la, \lt,{\lC}$) if and only if $1/\phi\in$WUSC($-\la, \lt,{1/\lC}$). Similarly, $\phi(t)\in$WLSC($\la, 0,{\lC}$) if and only if $\phi(1/t)\in$WUSC($-\la, 0,1/{\lC}$).
\end{remark}

\section{Scaling of the L\'evy-Khintchine exponent}\label{sec:SLKe}
We
shall
study consequences of
scaling of the L\'evy-Khintchine exponent $\psi$ of the (isotropic) unimodal pure-jump L\'evy process $X$ with nonzero L\'evy measure $\nu$.
We note that
$\psi$ always has global scalings with exponents $2$ and $0$, respectively. Inded, by
Remark~\ref{sch}, Lemma~\ref{sam} and Lemma~\ref{l:hop},
$\psi,\psi^*\in$WUSC($2$,$0$,$\pi^2$), and $\psi\in$WLSC($0$,$0$,${1/\pi^2}$).
In fact,
by \eqref{Psi*ScalingGeneral} we have
\begin{equation*}
\psi^*(u)\le\psi^*(\lambda u)\le 4 \lambda^2\psi^*(u),
\qquad \lambda\ge 1, u \ge 0,
\end{equation*}
and so $\psi^*\in$WUSC($2$,$0$,$4$). Of course, $\psi^*\in$WLSC($0$,$0$,$1$), meaning that $\psi^*$ is nondecreasing.
For economy of notation,
in the sequel we only consider (assume)
scaling exponents $\la$, $\ua$ satisfying:
\begin{equation}\label{02}
0<\la<2 \quad \text{ and } \quad 0<\ua<2.
\end{equation}
\kb{Under this convention we note  that $\psi\in$WLSC($\la$,$\lt$,$\lC$) for some $\lt\ge 0$, $\lC\in (0,1]$, if and only if the lower Matuszewska   index 
satisfies $\beta(\psi)>0$, and
$\psi\in$WUSC($\ua$,$\ut$,$\uC$) for some $\ut\ge 0$, $\uC\in [1,\infty)$, if and only if the upper Matuszewska   index   
satisfies $\alpha(\psi)<2$, see Remark~\ref{rem:Mi}.
Global scalings with \eqref{02} can also in principle be 
expressed in terms of Matuszewska indices.
Namely, by Remark~\ref{invitp}, Remark~\ref{SvC}, Proposition~\ref{Vestimate} and Lemma~\ref{scaling}, $\psi$ has weak global scaling (with exponent $\la>0$) if and only if
$\beta(\psi), \beta(1/\psi(1/t))>0$ (and we can take $\la=\min\{\beta(\psi), \beta(1/\psi(1/t))\}/2$). Similarly, the weak global upper scaling holds for $\psi$ (with $\ua<2$), if and only $\alpha(\psi), \alpha(1/\psi(1/t))<2$ (and we can take $\ua=1+\max\{ \alpha(\psi), \alpha(1/\psi(1/t))\}/2$).
}
\subsection{Examples}\label{sex}
The L\'evy-Khintchine (characteristic) exponents of unimodal convolution semigroups which we present in this section all have lower or upper scaling suggested by \eqref{02}. This can be verified in each case by using Lemma~\ref{sam}. While discussing the exponents, we shall also make connection to subordinators, special Bernstein functions and complete Bernstein functions, because they are intensely used in recent study of subordinate Brownian motions, a wide and diverse family of unimodal L\'evy processes cf. \cite{MR2986850}. The reader may find definitions and comprehensive information on these functions in \cite{MR2978140}.
When discussing subordinators we usually let $\varphi(\lambda)$ denote their Laplace exponent, and then $\psi(x)=\varphi(|x|^2)$ is the L\'evy-Khintchine exponent of the corresponding subordinate Brownian motion. We focus on scaling properties of $\psi$.
\begin{enumerate}
\item Let $\varphi(\lambda)=\int_0^\infty (1-e^{-\lambda u})\mu(dr)$ be a Bernstein function \cite{MR2978140}, i.e. the Laplace exponent of a subordinator $\eta$ \cite{MR2978140,MR1406564,MR1739520,MR2512800}, and let $Y$ be an independent (isotropic) unimodal L\'evy process with characteristic exponent
$\chi$. Then the
process $X_t= Y_{\eta_t}$ is unimodal and
has the
 characteristic exponent $\psi(x)=\varphi(\chi(x))$ \cite{MR2978140}. If $\chi\in$
WUSC$(\ua_1,\ut,\uC_1)$ and $\varphi\in$WUSC$(\ua_2,\chi(\ut),\uC_2)$,
then $\psi\in$WUSC$(\ua_1\ua_2,\ut,\uC_1^{\,\ua_2}\uC_2)$. From concavity of Bernstein functions it also follows that if $\chi\in$
WLSC$
(\la_1,\lt,\lC_1)
$, $\theta_*= \inf_{\theta\ge \lt} \chi(\theta)$ and $\varphi\in$WLSC$
(\la_2,\theta_*,\lC_2)$,
then $\psi\in$WLSC$
(\la_1\la_2,\lt,\lC_1\lC_2)$.
We always have $\theta_*\ge \chi(\lt)/{{\pi^2}}$, see Proposition~\ref{Vestimate},
and often $\theta_*= \chi(\lt)$.
Of particular interest here
is $
\chi({\xi})=|\xi|^2$, i.e. $Y_t=B_{2t}$, where $B$ is the standard Brownian motion in $\Rd$.
The process $X$ is then called a subordinate Brownian motion. Furthermore, it is called {\it special} subordinate Brownian motion if the subordinator is {\it special} (i.e. given by a special Bernstein function), and it is called {\it complete} subordinate Brownian motion if the subordinator is {\it complete} \cite{MR2978140}.
The (unimodal) L\'evy measure density of $Y$ is given by the formula
\begin{equation}\label{lmSBM}
\nu(x)=\int_0^\infty (4\pi t)^{-d/2}e^{-\frac{|x|^2}{4t}}\mu(dt),
\end{equation}
and its L\'evy-Khintchine exponent $\varphi(|{\xi}|^2)$ is in \kb{WUSC$(2\ua_2,\ut,\uC_2)$ or WLSC$(2\la_2,\tg{\lt,\lC_2})$}, respectively.

\item \label{ex:iM2} Let $\psi(\xi)=|\xi|^{\alpha}\log^{\beta}(1+|\xi|^{\gamma})$, where
$\gamma,\alpha,\alpha+2\beta\in (0,2)$.
If \mr{$0<\varepsilon< \min\{\alpha, 2-\alpha\} $}, then $\psi\in \textrm{WUSC}(\alpha+\varepsilon,1,
{\uC}
) \cap \textrm{WLSC}(\alpha-\varepsilon,1,\lC
)$ for some $0<\lC\leq 1\leq \uC<\infty$, 
\kb{and both Matuszewska indices of $\psi$ are equal to $\alpha$.}
Furthermore, $\psi\in \textrm{WUSC}(\alpha+
\gamma\beta_{{+}}
,0,1) \cap \textrm{WLSC}(\alpha{-}
\gamma\beta_{{-}}
,0,1)$. We note that $\psi$ is the L\'evy-Khintchine exponent of a subordinate Brownian motion, see Theorem 12.14, Proposition 7.10, Proposition 7.1, Corollary 7.9, Section 13 and examples 1 and 26 from Section 15.2 in \cite{MR2978140}. Many more examples related to subordinate Brownian motions readily follow from \cite[Section~15]{MR2978140}.

\item   Let $X$ be pure-jump unimodal with infinite L\'evy measure and L\'evy-Khintchine exponent $\psi$.
Let a
scaling condition with exponent $\la$ or $\ua$ hold for $\psi$.
For fixed $r>0$, we let $X^r$ be the (truncated) unimodal L\'evy process obtained by multiplying the L\'evy measure of $X$ by the indicator function of the ball $B_r$, and let $\psi_r$ be its L\'evy-Khintchine exponent. Since $0\le \psi-\psi_r$
is bounded, $\psi_r$ is comparable with $\psi$ at infinity, and so $\psi_r$ has (local) scaling with the same exponent as $\psi$. For later discussion we observe that $\psi_r$ is not an exponent of a subordinate Brownian motion because the support of its L\'evy measure is bounded \cite[Proposition 10.16]{MR2978140}.

\item 
\kb{We consider $\varphi(\lambda)=\int^{\infty}_0(1-e^{-r\lambda})\mu(dr)$, where $\mu$ is singular.
Namely, let  $\mu(dr)=\sum^\infty_{k=2}\delta_{1/k}(dr)\(k^{\alpha/2}-(k-1)^{\alpha/2}\)$
or $\mu(dr)=r^{-\gamma}F(dr)$, where $\alpha\in(0,2)$ and $\gamma=\alpha/2+\log 2/\log 3$ and $F$ is the standard Cantor measure on [0,1].
Such $\varphi$ is not complete or even special Bernstein function
\cite[Proposition 10.16]{MR2978140}.
In both cases we have
$\varphi(\lambda)\approx \lambda^{\alpha/2}\wedge \lambda$ (we use the integration by parts and \cite[Lemma 2]{GK2004} to verify the claim in the second case).
As usual, $\psi({\xi})=\varphi(|\xi|^2)
$ defines the characteristic exponent of a subordinate Brownian motion, and $\psi\in \textrm{WUSC}(\alpha,1,\uC) \cap \textrm{WLSC}(\alpha,1,\lC)$ for some $0<\lC\le 1\le \uC$.
}

\item Let $0<\alpha_1<\alpha_2<2$ and $u(r)=r^{\alpha_1/2-1}\vee r^{\alpha_2/2-1}$.
Let
$\eta(
\lambda)=\mathcal{L}u(\lambda)=\lambda^{-\alpha_1/2}\gamma(\alpha_1/2,\lambda)+\lambda^{-\alpha_2/2}\Gamma(\alpha_2/2,\lambda)$ and $\varphi(\lambda)=1/\eta(\lambda)$. Note that $\varphi(\lambda)\approx\lambda^{\alpha_1/2}\wedge\lambda^{\alpha_2/2}$.
Therefore
$\varphi(|x|^2)\in$
WUSC($\alpha_2,0,\uC$) and WLSC($\alpha_1,0,\lC$) for some $0<\lC\le 1\le \uC$.
It is shown in \cite[Example 10.18(i)]{MR2978140} that $\varphi$ is a special Bernstein function  but not a complete Bernstein function.
Moreover, the L\'{e}vy measure of $\varphi$ is not known and so previous methods of estimating transition densities of the resulting subordinate Brownian motion do not yet apply \cite{MR2986850}.

\end{enumerate}

\subsection{Estimates}
The following estimate is a version of \cite[Theorem 7 (ii) (b)]{MR2555291} with explicit constants.
\begin{lemma}\label{LaplaceLower}
Let $f\ge 0$ be
nonincreasing,
$\beta>0$ and
${\mathcal{L}f}\in \textrm{WUSC}(-\beta,\ut,\uC)$.
There is
${b}=b(\beta,\uC)\in(0,1)$ such that
\begin{equation*}f(r)\geq \frac b2e^{b} r^{-1}{\mathcal{L}f}(r^{-1}),\qquad 0<r<b/\ut .
\end{equation*}
\end{lemma}
\begin{proof}
Let $0<b<1$.
If ${u}>\ut$, then by Lemma \ref{LaplaceUpper} and the upper scaling (with $\lambda=s^{-1}/u$),
\begin{align*}
u\mathcal{L}f(u)&=u\int_0^{bu^{-1}}e^{-us}f(s)ds+u\int_{bu^{-1}}^\infty e^{-us}f(s)ds\\
&\leq \frac{u}{\gamma(1,1)}\int^{bu^{-1}}_0e^{-
{us}}\mathcal{L}f(s^{-1})s^{-1}ds+f(
bu^{-1})\int_{bu^{-1}}^\infty e^{-us}uds\\
&\leq \frac{u}{\gamma(1,1)}\int^{bu^{-1}}_0\uC\(us\)^{\beta}\mathcal{L}f(
{u})e^{-us}s^{-1}ds+f(bu^{-1})e^{-b}\\
&= \uC\frac{\gamma(\beta,b)}{\gamma(1,1)}u\mathcal{L}f(
{u})+f(bu^{-1})e^{-b}.
\end{align*}
{If}
$2\uC\gamma(\beta,b)
\leq {\gamma(1,1)}=1-e^{-1}$, then
$
f(bu^{-1})\geq e^b u\mathcal{L}f(u){/2}$.
We change variables: $r=bu^{-1}$. Since $\mathcal{L}f$ is decreasing,
$$f(r)\geq \frac{b}{2}e^{b}r^{-1}\mathcal{L}f(br^{-1})\geq
\frac{b}{2}e^{b}r^{-1}\mathcal{L}f(r^{-1}), \qquad r<b/\ut.$$
\end{proof}

\begin{lemma} \label{tail}
$C=C(d)$ exists such that if $\psi{\in}$
WUSC$(\ua,\ut,\uC)$ and $a=[(2-\ua)C]^{\frac2{2-\ua}}\uC^{\frac{\ua-2}{2}}$,  then
$${\p}  (|X_t|\ge r)\geq a \(1-e^{-t\psi^*(1/r)}\), \quad \quad 0<r< \sqrt{a}\big/\ut.
$$
\end{lemma}
\begin{proof}
By Lemma \ref{laplace}, Proposition \ref{Vestimate} and \eqref{nfw},
for $\lambda\ge 1$ and $u\ge \ut^2$ we have
$$\frac{\mathcal{L}{f}_t(\lambda u)}{\mathcal{L}{f}_t(u)}\le C_1^2 {\lambda}^{-1}\frac{1-e^{-\pi^2t\psi(\sqrt{\lambda u})}}{1-e^{-t\psi(\sqrt{u})}}\leq C_1^2 {\lambda}^{-1}\frac{1-e^{-\pi^2t\psi( \sqrt{u})\uC\lambda^{\ua/2}}}{1-e^{-t\psi(\sqrt{u})}}\leq \pi^2C_1^2 \uC\lambda^{\ua/2-1}.$$
Thus, $\mathcal{L}{f}_t\in$WUSC$(\ua/2-1,\ut^{\,2},\pi^2C_1^2\uC)$.
By Lemma \ref{LaplaceLower}
and Lemma~\ref{laplace},
\begin{equation*}
\p(|X_t|\ge r)=f_t(r^{2})\geq \frac {b}{2}e^br^{-2}
\mathcal{L}f_t(
r^{-2})\geq \frac{b}{2C_1}
\(1-e^{-t\psi^*(
1/r
)}\),\qquad r^{2}< b/{\ut}
^2.
\end{equation*}
Here $b\in(0,1)$ is such that $2\pi^2C_1^2\uC\gamma(1-\ua/2,b)\leq 1-e^{-1}$, see
the proof of Lemma \ref{LaplaceLower}.
Since $\gamma(1-\ua/2,b)<b^{1-\ua/2}/(1-\ua/2)$, we may take $b=\(\frac{1-\ua/2}{
2\pi^2C_1^2\uC}\)^{1/(1-\ua/2)}$and $a=b/(2C_1)<1$.
\end{proof}
Since $\lim_{t\to 0^+}\p (|X_t|\ge r)/t=\nu(B_r^c)=L(r)$ for $r>0$, we obtain the following result.
\begin{corollary}\label{GApprox}
If $\psi$ satisfies WUSC$(\ua,\ut,\uC)$ and
$a$ is from Lemma \ref{tail}, then
$$L(r)\geq a \psi^*(1/r
), \qquad {0<}r< \sqrt{a}/\,\ut.$$
\end{corollary}
We recall that a reverse inequality is valid \mr{for every unimodal process}, cf. \eqref{tails}.

\noindent
\kb{The following general lemma will be useful in Fourier inversion.}
\begin{lemma} \label{inverse}
\mr{Let  $\alpha>0$, $0<K<\infty$ and $\Psi\in$WLSC$({\alpha}, \lt, \lC)$ be  an increasing function on $[0,\infty)$ with $\Psi(0)=0$.
There is ${C}=C(d,\alpha, \lC, K)$ such that if  
\kb{$0<t<K/\Psi(\lt)$}, then
$$\int_{\R^d}e^{ -t \Psi(|\xi|)}d\xi  \le
C  \[\Psi^{-1}( K /t)\]^{d}.$$}
\end{lemma}
\begin{proof}
\kb{We note that the condition $t<K/\Psi(\lt)$ is nonrestrictive if $\Psi(\lt)=0$.}
Since $\Psi$ increases and scales,
it is unbounded. Hence, {for $t>0$},
\begin{eqnarray*}\int_{\R^d}e^{ -t \Psi(|\xi|)} d\xi &=& \sum_{n=1}^\infty  \int_{\lC K (n-1)\le t\Psi(|\xi|)<\lC K  n}  e^{ -t \Psi(|\xi|)}d\xi\le
 {\frac{\omega_d}{d}}\sum_{n=1}^\infty  [\Psi^{-1}(\lC K  n/t)]^{d}  e^{ -\lC K (n-1)}.\end{eqnarray*}
Also, if
$\Psi(\lt)
< K/t$, then
$\Psi^{-1}( K /t)
> \lt$.
By lower scaling, for $n\ge 1$,
\begin{equation*}
 \Psi(n^{1/\alpha}\Psi^{-1}( K /t)) \ge
\lC n\Psi\(\Psi^{-1}( K /t)\)
= \lC n K /t,\end{equation*}
which yields
\begin{equation*} n^{1/\alpha}\Psi^{-1}( K /t)
\ge
\Psi^{-1}(\lC K  n/t).\end{equation*}
We obtain
\begin{eqnarray*}\int_{\R^d}e^{ -t \Psi(|\xi|)} d\xi &\le&
[\Psi^{-1}( K /t)]^d\;{\frac{\omega_d}{d}}\sum_{n=1}^\infty   n^{d/\alpha}  e^{ -\lC K (n-1)}.\end{eqnarray*}
\mr{Taking $C= {{\omega_d}{d^{-1}}}\sum_{n=1}^\infty   n^{d/\alpha}  e^{ -\lC K (n-1)}<\infty$, we complete the proof.}
\end{proof}
\mr{
\begin{remark}\label{constant1} It is desirable for future applications to specify how $C$ in Lemma \ref{inverse} depends 
on $K$ and $\lC$. 
To this end for $\rho\ge 0$ and $u>0$ we consider
\begin{eqnarray*}
S(u, \rho)&=&\sum_{n=1}^\infty  n^{\rho}  e^{ -u(n-1)}= e^{ 2u}\sum_{n=1}^\infty  n^{\rho}  e^{ -u(n+1)} \\&\le& e^{ 2u} \int_0^\infty x^{\rho}e^{-ux}dx= e^{ 2u}u^{-1-\rho}\Gamma(\rho+1).
\end{eqnarray*}
Since $S(u, \rho)$ is decreasing in $u$, we also have $S(u, \rho)\le S(1, \rho)$ for $u\ge 1$, thus
$$S(u, \rho)\le  e^{ 2}\Gamma(\rho+1) (1\vee u^{-1-\rho}), \qquad u>0.$$
  Therefore,
$C={{\omega_d}{d^{-1}}}\sum_{n=1}^\infty  n^{d/\alpha}  e^{ -\lC K (n-1)}\le  \omega_d d^{-1}e^{ 2}\Gamma(d/\alpha+1)
{(1\vee (\lC K )^{-d/\alpha-1})}$.
\end{remark}}

For a continuous nondecreasing function $\phi : [0,\infty) \to [0,\infty)$, {such that $\phi(0)=0$,} we let $\phi(\infty)=\lim_{s\to \infty}\phi(s)$ and
define
the generalized inverse $\phi^{-}: {[0,\infty) \to [0, \infty)}$,
\begin{equation*}\label{2s}
\phi^{-}({u}) = \inf \{{s}
\ge 0 : \phi (s)\ge u\},\quad 0\le u<\infty,
\end{equation*}
with the convention that $\inf \emptyset = \infty$. The function $\phi^{-}$ is nondecreasing and c\`agl\`ad \mr{(left continuous with right-hand side  limits)  on the set where it is} finite.
Notice that $\phi(\phi^{-}(u))=u$ for $u\in [0,\phi(\infty)]$ and $\phi^{-}(\phi(s))\leq s$ for $s\in [0,\infty)$.
Also, if $\varphi : [0,\infty) \to [0,\infty)$, $\varphi(0)=0$, $c>0$ and $c\phi\le \varphi$, then $\phi^-(u)\ge \varphi^-(cu)$, $u\ge 0$.
Below we often consider the (unbounded) characteristic exponent $\psi$ of a unimodal L\'evy process with infinite L\'evy measure and its (comparable) maximal function $\psi^*$,
and denote $$\psi^-=(\psi^*)^-.$$ This short notation is motivated by
the following equality:
\begin{equation*}\label{2sp}
\inf \{{s}
\ge 0 : \psi (s)\ge u\}=\inf \{{s}
\ge 0 : \psi^* (s)\ge u\},\qquad 0\le u<\infty.
\end{equation*}
Note that $\psi^*(\psi^{-}(u))=u$,  $\psi^{-}(\psi^*(s))\leq s$.
The reader may find it instructive to prove the following result for $t>0$ and $x\in \Rdz$.
\begin{lemma}\label{ltst}
$t\psi^*(1/|x|)\ge 1\;$
if and only if
$\;t\psi^*(1/|x|)|x|^{-d}\ge \[\psi^-(1/t)\]^d$.
\end{lemma}

In what follows it may be helpful to view $t\psi^*(1/|x|)\ge 1$ and $t\psi^*(1/|x|)< 1$ as defining ``large time" and ``small time" for given $x\in \Rdz$, respectively.

We note that $\psi^{-}(\psi^*(s)+)\ge s$ for $
s\in [0,\infty)$,
where $\psi^{-}(u+)$ denotes the right hand side limit of  $\psi^{-}$  at $u\ge0$. Furthermore, scaling of $\psi$ translates into scaling of $\psi^-$ as follows.
\begin{lemma}\label{scpsii}
If $\psi\in\textrm{WLSC}(\la,0,\lC)$, then $\psi^-\in \textrm{WLSC}\(1/2,0,(\lC/\pi^4)^{1/\la}\)\cap\textrm{WUSC}\(1/\la,0,(\pi^3/\lC)^{2/\la}\) $.
\end{lemma}
\begin{proof}
We let
$\Psi(s)= h_1(s^{-1})$, $s>0$. Lemma \ref{sam} and Remark~\ref{sch} yield $\Psi\in \textrm{WUSC}(2,0,1)$.
By Lemma \ref{l:hop} and Remark~\ref{SvC},
$\Psi\in \textrm{WLSC}(\la,0,\lC/\pi^2)$.
Remark~\ref{inverseScal} implies
$$\Psi^{-1} \in \textrm{WLSC}(1/2,0,1)\cap\textrm{WUSC}(1/\la,0,(\pi^2/\lC)^{1/\la}).$$
By Lemma \ref{l:hop} and the above scaling,
$$\Psi^{-1}(s/2)\leq \psi^-({s}
)\leq \Psi^{-1}(s\pi^2/2)\leq (\pi^4/\lC)^{1/\la}\Psi^{-1}(s/2).$$
Hence, by Remark~{\ref{SvC}},
$\psi^-\in \textrm{WLSC}(1/2,0,(\lC/\pi^4)^{1/\la})\cap\textrm{WUSC}(1/\la,0,(\pi^3/\lC)^{2/\la})$.
\end{proof}
We shall use Fourier inversion and \eqref{LKf} to estimate $p_t(0)$: if, say,
$\lim_{|\xi|\to\infty}\psi({\xi})/\ln|\xi|=\infty$,
then  $e^{-t\psi(\xi)}$ is integrable for $t>0$, and 
\begin{align*}
p_t(x)&={(2\pi)^{-d}}\int_{\Rd}e^{ -t \psi({\xi})}e^{-i\langle \xi,x\rangle} d\xi
\end{align*}
is 
continuous and 
bounded in $x\in \Rd$ together with all its derivatives.
In particular,
\begin{align}
p_t(0)&
\geq {(2\pi)^{-d}}\int_{B(0,{ \psi^{-}(1/t)})} e^{-t\psi^*(|\xi|)}d\xi
\geq {(2\pi)^{-d}}\frac{\omega_d}{ed}\[\psi^{-}(1/t)\]^{d},\qquad t>0.\label{lbpt}
\end{align}
Lower scaling yields a reverse inequality.
\begin{proposition}\label{sup_p_t_psi}If $\alpha>0$ and $\psi\in$
WLSC$(\alpha,\lt,{\lC})$, then
$C=C(d,\alpha, \lC)$
exists such that
\begin{equation}\label{bptp}
p_t(
x)   \le
\mr{C}\[ \psi^{-}(1/t)\]^{d} \qquad
\text{{if} }\; {t>0} \; \text{ and }\;
t\psi^*({\lt})
< 1/\pi^2.
\end{equation}
In fact, $C=c\lC^{-d/\alpha-1}$, where $c=c(d,\alpha)$.
\end{proposition}
\begin{proof}
We let
$$\Psi(s)= h_1(s^{-1}),\qquad s>0.$$ Note that $h_1$ and $\Psi$ are strictly monotone.
According to Lemma \ref{l:hop},
$$2\Psi(s)/\pi^2\le \psi(s)\le \psi^*(s)\le 2\Psi(s),$$
hence $\Psi\in$
WLSC$(\alpha,\lt,{\lC}/\pi^2)$.
Furthermore,
$\psi^-(2u/\pi^2)\le \Psi^{-1}(u)\le \psi^-(2u)$ for $u\ge 0$.
Let $t>0$. If
$
t\psi^*({\lt})
< 1/\pi^2$, then
$(2t/\pi^2) \Psi(\lt)
< 1/\pi^2$. We  apply
Lemma \ref{inverse} \mr{with constant 
${K}=1/\pi^2$ and
 $2t/\pi^2$ instead of $t$ and we obtain}
$$
p_t(
x)\le  {(2\pi)^{-d}}\int_{\Rd}e^{ -2t  \Psi(|\xi|)/\pi^2}d\xi   \le
\mr{C(d,\alpha, \lC)}
 \[\Psi^{-1}((2t)^{-1})\]^{d}.$$
By Remark \ref{constant1}, $C(d,\alpha, \lC)=c(d,\alpha)\lC^{-d/\alpha-1}$.
But $\Psi^{-1}((2t)^{-1})\leq \psi^-(1/t)$, ending the proof.
\end{proof}
Under the assumptions of Proposition~\ref{sup_p_t_psi}, by \eqref{lbpt} and \eqref{bptp} we obtain
\mr{\begin{equation}\label{apt0}
p_t(0)\approx \[\psi^{-}(1/t)\]^d
\qquad
\text{{if} }\; t>0,\;
t\psi^*({\lt})
< 1/\pi^2,
\end{equation}
where the comparability constant depends only on $\lC,\, \alpha$ and $d$. This allows  to interchange $p_t(0)$ and $[\psi^-(1/t)]^d$ in approximation formulas below.}

Also, if {$\alpha>0$} and $\psi\in$
WLSC$(\alpha,0,{\lC})$,
then
for
$t>0$, {$x\in \Rd$},
\mr{\begin{equation}\label{bptpz}
p_t(
{x})   \le
C{(d,\alpha, \lC)} \[ \psi^{-}(1/t)\]^{d}.
\end{equation}}
We note in passing that the same argument covers the Gaussian case  $\psi(\xi)=|\xi|^2$ and more general exponents
otherwise excluded from our general considerations. We also note that
analogues
of \eqref{bptpz}
are often
obtained by using Nash inequalities
\cite{MR2299447,MR898496,MR2492992}.

\begin{corollary}
If
$\psi\in$WLSC($\la,\lt,\lC$) with $\lt>0$, and $0<T<\infty$,
then $p_t(
{x})   \le
{C} \[ \psi^{-}(1/t)\]^{d}$ for all $0<t< T$ and $x\in \Rd$, with
 \mr{$C=C(d,T,\la,\lt,\lC,\psi(\lt))$}.
\end{corollary}
\begin{proof}
Note that $\psi^*$ satisfies WLSC($\la,\lt,\lC/{{\pi^2}}$). Since $\psi^*(0)=0$ and $\psi^*$ is continuous and unbounded, there is $\theta_1>0$ such that $\psi^*(\theta_1)=1/(\pi^2 T)$.
By Lemma \ref{scaling}, $\psi^*$ satisfies WLSC($\la,
\lt\wedge \theta_1, c_1 $), with a constant $c_1$.
An application of
Proposition~\ref{sup_p_t_psi} completes the proof.
In fact, if $\psi\in WLSC(\la,\lt,\lC)$,  
$\theta_1<\lt$, then
$1/(\pi^2 T)=\psi^*(\theta_1)\le \psi(\lt)(\lC/{{\pi^2}})^{-1}(\lt/\theta_1)^{-\la}$,
hence $\theta_1^\la\ge \lC \lt^\la/[{\pi^4}T\psi(\lt)]$,
leading to $C=C(T,{\la,\lt,\lC,}\psi{(\lt)})$.
\end{proof}
Thus, \eqref{bptpz} holds for all $t>0$, even if $\lt>0$, but the constant deteriorates for large $t$.

The following
main result of our paper gives {\it common bounds}
for unimodal convolution semigroups with scaling.
Notably, our second main result, Theorem~\ref{NWSR} below, shows in addition that scaling is equivalent to common bounds.
\begin{theorem}\label{densityApprox}If  $\psi\in$WLSC$(\la,{\theta},\lC)$,  then there is \mr{$
{C^*}=
C^*(d,\la, \lC)$}
such that
$$p_t(x)   \le
\mr{{C^*}} \min\left\{  \[ \psi^{-}(1/t)\]^{d},\;\frac{t\psi^*(1/|x|)}{|x|^d}\right\} \quad\text{{if} }\; {t>0} \; \text{ and }\;
t\psi^*({{{\theta}}})
< 1/\pi^2.$$
If $\psi\in$WLSC$(\la,{{\theta}},\lC)\cap$WUSC$(\ua,{{\theta}},\uC)$, then
$
{c^*=c^*}(d,\la,\lC,\ua,\uC)$, $r_0
=r_0(d,\la, \lC,\ua, \uC)$
exist such that
$$p_t(x) \ge c^* \min\left\{ \[  \psi^{-} \(1/t\)  \]^{d},\;\frac{t\psi^*(1/|x|)}{|x|^d}\right\} \qquad \text{if}\quad t>0,\quad  t\psi^*(
{{{\theta}}}/r_0)< 1 \quad \text{and}\quad |x|< r_0/{{\theta}}.$$
\end{theorem}
\begin{proof}
Let $t>0$. For $x=0$ the term $t\psi^*(1/|x|)/|x|^d$ in the statement should be ignored--the bounds are to be understood as \eqref{lbpt} and \eqref{bptp}. Accordingly, below we let $x\in \Rdz$.
The upper bound now
follows from Corollary \ref{upper_den} and
Proposition
~\ref{sup_p_t_psi}.

To prove the lower bound, we take
$\kappa\ge 2$.
 We have
$$p_t(x)\ge
\frac{
{\p}
 \(  |x|\leq |X_t|<\kappa |x|\)}{ \left|B_{\kappa|x|}\setminus B_{|x|}\right|}
= \frac d{\omega_d(\kappa^d-1)}  |x|^{-d}\(\p (   |X_t|\ge |x|) - \p (   |X_t|\ge \kappa |x|)\).$$
Let $|x|< {\sqrt{ a}}/{{\theta}}$, with $a=a(d,\ua,\uC)$
from Lemma \ref{tail}.
We now suppose that $t\psi^*(1/|x|)\leq 1$.
By concavity, $s/2\le 1-e^{-s}\le s$ for $0\leq
{s}\le 1$.
By Lemma \ref{tail} and  Corollary \ref{ubt},
\begin{eqnarray*}\p (   |X_t|\ge |x|) - \p (   |X_t|\ge \kappa |x|)
&\ge& \frac{a}{2}t\psi^*(1/|x|)-
\frac{2e}{e-1}(2d+1)
t\psi^*(1/|\kappa x|)\\
&\ge&\frac{a}{2}t\psi^*(1/|x|)
\(
1-\frac{4e(2d+1)}{a(e-1)}\frac{\psi^*(1/|\kappa x|)} {\psi^*(1/|x|)}\).
\end{eqnarray*}
Recall that $\psi^*\in$WLSC$(\la,{{\theta}},\lC/{{\pi^2}})$. 
\kb{We now take}
\begin{equation}\label{dkappa}
\kappa=\big(8{{\pi^2}}e(2d+1)/(\lC a (e-1))\big)^{1/\la}.
\end{equation}If $|x|
< 1/(\kappa{{\theta}})$, then
$$
\frac{4e(2d+1)}{a(e-1)}\frac{\psi^*(1/|\kappa x|)} {\psi^*(1/|x|)}\le
\frac{4{{\pi^2}}e(2d+1)}{\lC a(e-1)}\kappa^{-\la}=\frac12,
$$

Recall that $\lC, a\in (0,1]$ (see the proof of Lemma~\ref{tail}), so $\kappa\ge 2$, as required.
\kb{We also have} $\kappa^{-1}<\sqrt{a}$.
Therefore,
\begin{eqnarray}\label{lower_h}
&p_t(x) &\geq \frac{d}{\omega_d (\kappa^d-1)} \frac{a}{4} t \psi^* (1/|x|) |x|^{-d}\\
&& \geq \frac{ad}{4\omega_d (\kappa^d-1)}\min \left\{[\psi^{-}(1/t)]^d , \frac{t\psi^*(1/|x|)}{|x|^d}\right\}\nonumber.
\end{eqnarray}

We are in a position to verify that the lower bound in the statement of the theorem holds with $r_0=\min\{\kappa^{-1},\sqrt{a}\}=\kappa^{-1}$.
We thus assume that $t>0$, $t\psi^*(
{{{\theta}}}/{r_0})
< 1$, and because of the preceding discussion we only need to resolve the case $0<|x|< r_0/{{\theta}}$, $t\psi^*(1/ |x|)
\ge 1$. By continuity, there is $x^*\in \Rd$ such that
$|x|
\le|x^*|
< r_0/{{\theta}}$ and $t\psi^*(1/ |x^*|)=1$. By \eqref{lower_h},
\begin{align*}p_t(x)&\ge p_t(x^*)\ge
{\frac {ad}{4\omega_d{\kappa^d}}
}\frac{1} {|x^*|^{d}}\ge
{\frac {ad}{4\omega_d{\kappa^d}}}\[\psi^{-}(1/ t)\]^d\\
&
\ge \frac {ad}{4\omega_d{\kappa^d}}
\min\left\{\[\psi^{-}(1/t)\]^d,\frac{t\psi^*(1/|x|)} {|x|^{d}}\right\}.
\end{align*}
This ends the proof and we may take $c^*=ad\kappa^{-d}/(4\omega_d)$.
\end{proof}
\begin{remark}
For the record we note that the constants \mr{ in the upper and lower bounds}
depend on $d$, $\la$, $\lC$, $\ua$, $\uC$ via \mr{Proposition~\ref{sup_p_t_psi}, \eqref{dkappa} and Lemma~\ref{tail}, e.g.
$C^*= c(d,\la)(\lC)^{-d/\la-1}$},  $c^*=c(d,\la,\ua)\,\lC^{d/\la}\,\uC^{(\ua-2)(d+\la)/(2\la)}$ and
$r_0=
c(\la,\ua, d)\(\uC^{\frac{\ua-2}2}\lC\)^{1/\la}$.
\end{remark}
In view of Lemma~\ref{ltst}, the two factors in the minima in the statement of Theorem~\ref{densityApprox} should be interpreted as the approximations of $p_t(x)$ in large time (on-diagonal regime) and small time (off-diagonal regime), correspondingly.

We emphasize that the upper bound in Theorem~\ref{densityApprox} only requires the lower scaling. For instance the upper bound holds for the $2$-regularly varying characteristic exponent $\psi(\xi)=|\xi|^2/[\log(1+|\xi|^{2})]^\beta$ with $\beta\in (0,1)$, which is in agreement
with the outcome of the Davies' method in this case \cite{MR2886459}.
\kb{Note that in principle we can track 
constants in our estimates,  
 see \cite[(29)]{2013arXiv1305.0976B} for the isotropic $\alpha$-stable L\'evy process .}
We now list a number of {general} consequences of Theorem~\ref{densityApprox}.
We first complement \eqref{onuzg} by a similar lower bound resulting from Theorem~\ref{densityApprox}.
\begin{corollary}\label{nuApprox}
If $\psi\in$WLSC$(\la,
{\theta},\lC)\cap$WUSC$(\ua,
{\theta},\uC)$ and $|x|< r_0/
{\theta}$, then
$\nu(x)
\ge c^* \psi^*(1/|x|)|x|^{-d}$.
\end{corollary}

\begin{corollary}\label{cc}
If $\psi\in$
WLSC($\la$,$0$,$\lC$)$\cap$WUSC($\ua$,$0$,$\uC$), then \eqref{allcomp} holds for all $t>0$ and $x\in \Rd$.
\end{corollary}
\begin{proof}
We use \eqref{onuzg} and Corollary~\ref{nuApprox} to obtain
\begin{equation}\label{aml}
\nu({x})\approx \frac {\psi^*(|x|^{-1})}{|x|^d},\qquad\quad x\in \Rdz.
\end{equation}
We then appeal to \eqref{lbpt}, \eqref{bptpz} and Theorem~\ref{densityApprox}.
\end{proof}
By scaling, in particular by Lemma~\ref{scpsii}, we obtain the following important doubling property, cf. \cite{MR2925579} in this connection.
\begin{corollary}\label{Doubling}
If $\psi$ satisfies (global)
WLSC($\la$,$0$,$\lC$) and WUSC($\ua$,$0$,$\uC$), then
$$p_t(2x)\approx p_t(x)\quad  \text{ and } \quad p_{2t}(x)\approx p_t(x),\qquad\quad t>0,\, x\in \Rd.$$
\end{corollary}

Thus, if $
{\theta}=0$ in Theorem~\ref{densityApprox}, then
the global asymptotics of $p_t(x)$ is fully and conveniently reflected by $\psi$.
If
$
{\theta}>0$, then our bounds are only guaranteed to hold in bounded time and space (bounded time for the upper bound). For large times we merely offer the following simple exercise of monotonicity.
\begin{corollary}\label{density_0}
If $\psi\in$WLSC$(\la,\lt,\lC)$, $0<|x|
< \(
\pi^{-4}\lC\)^{1/\la}/\lt$ and $t\psi^*({|x|^{-1}})\ge 1$,
then \tg{$C=C(d,\la,\lC)$ exists such that }
\begin{equation*}\label{pvnu}
p_t(0)   \le
C t\psi^*(|x|^{-1})/|x|^d.
\end{equation*}
\end{corollary}
\begin{proof}
Define (threshold time) $t_0=1/\psi^*(|x|^{-1})$.
By Proposition \ref{Vestimate}, $\psi^*\in$WLSC$(\la,\lt,\lC/\pi^2)$, thus $$t_0{\psi^*(\lt)}=\frac{\psi^*(\lt)}{\psi^*(|x|^{-1})}
< \pi^{-2}.$$
Since $t\to p_t(0)$ is decreasing, by
Proposition
~\ref{sup_p_t_psi} we have for $t\ge t_0$,
$$p_t(0)\leq p_{t_0}(0) \leq
\tg{C} \[ \psi^{-}(\psi^*(|x|^{-1}))\]^{d}\le
\tg{C}\frac{1}{|x|^d}\leq \tg{C} \frac{t\psi^*(|x|^{-1})}{|x|^d},$$
which completes the proof.
\end{proof}

The next theorem proves that our definitions quite capture the subject of the study.
\begin{theorem}\label{NWSR}
Let $X_t$ be an isotropic unimodal L\'{e}vy process in $\Rd$ with transition density $p$, L\'evy-Khintchine exponent $\psi$ and L\'evy measure density $\nu$. The following
are equivalent:
\begin{description}
	\item[{\it (i)}]WLSC and WUSC 
	hold for $\psi$.
		\item[\it (ii)]\kb{There are $r,c>0$,}
such that
\begin{align*}
p_t(x)&\geq \mr{c} \frac{t\psi^*(|x|^{-1})}{|x|^d}, \qquad 0<|x|< \mr{r},\; 0<t\psi^*(|x|^{-1})< 1.\end{align*}
	\item[\it (iii)]\mr{There are} $r,
c>0$, such that
\begin{align*}
\nu(x)&\ge c \frac{\psi^*(|x|^{-1})}{|x|^d}, \qquad 0<|x|< r.\quad \qquad\qquad \qquad
\end{align*}
\end{description}
\mr{If we instead assume global  WLSC and WUSC in (i), and let $r=\infty$  in  (ii) and (iii), then the three conditions are equivalent, too. }
\end{theorem}
\begin{proof}
Theorem \ref{densityApprox} and Lemma {\ref{ltst}} yield the implication $(i)\Rightarrow (ii)$.
The implication $(ii)\Rightarrow (iii)$ follows because $\lim_{t\to 0^+}p(t,x)/t=\nu(x)$ vaguely on $\Rdz$.
 To prove that $(iii)$ implies $(i)$, we assume that $(iii)$ holds.  By the  L\'evy-Khintchine formula, $\psi(x)=\sigma|x|^2+\int_\Rd (1- \cos\langle \xi,x\rangle)\nu(dx)$, where $\sigma\ge 0$.
Actually, we must have $\sigma=0$, because
$$\infty>\int_{B_1}|x|^2\nu(x)dx\ge \int_{B_{1\wedge r}} c |x|^2\frac{\psi(|x|^{-1})}{|x|^d}dx\ge
\sigma \int_{B_{1\wedge r}} \frac{c}{|x|^d}dx
.$$
By \cite[proof of Theorem 6.2]{MR2978140} and \eqref{wml}, the following defines a {\it complete Bernstein function}:
$$\varphi(\lambda)=\Z \frac{\lambda}{\lambda +s}s^{-1}\nu(s^{-1/2})s^{-d/2}ds=\Z \(1-e^{-\lambda u}\)\mu(u)du,\quad \lambda\geq 0,$$
where $\mu(u)=\mathcal{L}[\nu(s^{-1/2})s^{-d/2}](u)$. In fact, by changing variables,
and \eqref{def:GKh} for $\lambda>0$ we have
\begin{align*}\label{defphil}
\varphi(\lambda)&=2\Z \frac{\lambda u^2}{\lambda u^2+1}\nu(u)u^{d-1}du\approx \Z \[1\wedge (\lambda u^2)\]\nu(u)u^{d-1}du=\omega_d^{-1}h\(\lambda^{-1/2}\).
\end{align*}
By Corollary \ref{ch1Vpc},
 there exists $c_1=c_1(d)$
such that
\begin {equation}\label{comp10}c_1 \varphi(
\lambda
) \leq\psi\(\sqrt{\lambda} \)\leq c_1^{-1} \varphi(\lambda
), \quad \lambda\ge 0.
\end{equation}
Since $\varphi'(\lambda)=\Z ue^{-\lambda u}\mu(u)du$ and $\mu$ is decreasing, Lemma \ref{LaplaceUpper} with $n=0$ and $m=1$ yields
\begin{equation}\label{NWSR1}\mu(u)\leq \frac{1}{\gamma(2,1)}\frac{\varphi'(u^{-1})}{u^2},\quad u>0.\end{equation}
Using  the upper incomplete gamma function and monotonicity of $\nu$, we obtain
\begin{equation}\label{NWSR2}\nu(x)\leq \frac{\int^\infty_{|x|^{-2}}e^{-s |x|^{2}}\nu(s^{-1/2})s^{-d/2}ds}{\int^\infty_{|x|^{-2}}e^{-s |x|^{2}}s^{-d/2}ds}\leq \frac{\mu(|x|^2)}{|x|^{d-2}
\Gamma(1-d/2,1)}, \quad x\neq 0.   \end{equation}

We leave it at that for a moment, to make another observation.
As $\varphi$ is a complete Bernstein function, we have that $$\varphi_1(\lambda)=\lambda/\varphi(\lambda)$$ is a special Bernstein function  (see \cite[Definition 10.1 and Proposition 7.1]{MR2978140}).
Since $X_t$ is pure-jump, $\lim_{|\xi|\to \infty } \psi({\xi})/|\xi|^2=0$. Thus, $\lim_{\lambda\to \infty}\varphi_1(\lambda)=\infty$.
Also,
$\varphi(0)=0$, and by \cite[(10.9) and  Theorem 10.3]{MR2978140}, the potential measure of the subordinator with the Laplace exponent
$\varphi_1$ is absolutely continuous with the density function
\begin{equation*}\label{e:deff}
f(s)=
\int_s^\infty\mu(u)du.
\end{equation*}
In particular $\mathcal{L}f=1/\varphi_1$.
Let $x\in \Rd$ be such that $0<|x|< r$. By \eqref{comp10}, $(iii)$,  \eqref{NWSR2} and \eqref{NWSR1},
\begin{align}\label{dcn}
&c_1c \frac{\varphi(|x|^{-2})}{|x|^d}
\leq c \frac{\psi(|x|^{-1})}{|x|^d}
\le \nu(x)
\le \frac{\mu(|x|^2)}{|x|^{d-2}\Gamma(1-d/2,1)}
\le
\frac{\varphi'(|x|^{-2})}{|x|^{d+2}\Gamma(1-d/2,1)\gamma(2,1)}.
\end{align}
Therefore
 $c_2\varphi(\lambda)\leq \lambda\varphi'(\lambda)$ for $\lambda> 1/\sqrt{r}$, where $c_2=cc_1\gamma(2,1)
\Gamma(1-d/2,1)$. This implies that the function $\lambda^{-c_2}\varphi(\lambda)$ is nondecreasing on $(1/\sqrt{r},\infty)$. By Remark~\ref{sam}, and \eqref{comp10},
$\varphi\in\textrm{WLSC}(c_2,1/\sqrt{r}
,1)$.
Hence, $\psi\in$WLSC$(2c_2,r^{-1},c_1)$.

We now prove the upper scaling.
By concavity of $\varphi$, $u\varphi'(u)\leq \varphi(u)$. For $0<s<\sqrt{r}$, by
\eqref{dcn},
$$f(s)\geq c_3 \int^{\sqrt{r}}_s\varphi(u^{-1})u^{-1}du
\geq c_3 \int^{\sqrt{r}}_s\varphi'(u^{-1})u^{-2}du = c_3\(\varphi(s^{-1})-\varphi(r^{-1/2})\).$$
Note that $\varphi$ is strictly increasing. Therefore, for $0<s<\sqrt{r}/2$, we get
\begin{equation}\label{oodfpp}
f(s)\geq c_4  \varphi(s^{-1})=c_4 /[s\varphi_1(s^{-1})].
\end{equation}
Since $f$ is decreasing and $\mathcal{L}f(u)=1/\varphi_1(u)$, by Lemma \ref{LaplaceUpper} we obtain
$$ f(s)\leq \frac{1}{\gamma(2,1)s^2}\(\frac{1}{\varphi_1}\)'(s^{-1})
=\frac{1}{\gamma(2,1)s^2}\frac{\varphi_1'(s^{-1})}{\varphi_1^2(s^{-1})}.$$
Thus, by \eqref{oodfpp}, $c_5\varphi_1(\lambda)\leq \lambda\varphi'_1(\lambda)$, where $\lambda> 2/\sqrt{r}$ and $c_5=c_4 \gamma(2,1)$.
It follows as above that $\varphi_1\in\textrm{WLSC}(c_5, 2{r}^{-1/2},1)$. Since $\varphi_1$ is concave, $\lambda\varphi_1'(\lambda)\leq \varphi_1(\lambda)$,
hence
$c_5<1$.
Thus, $\varphi\in\textrm{WUSC}(1-c_5,2{r}^{-1/2},1)$. This implies $\psi\in\textrm{WUSC}(2(1-c_5),4r^{-1},c_1^{-2})$.
\end{proof}
The reader may consult, e.g., Example 3 in Section~\ref{sex} for typical asymptotics of $p_t(x)$ and $\nu(x)$.
Our next observation results from Corollary \ref{nuApprox}, \eqref{aml}, \eqref{comp10} and
Theorem \ref{NWSR}.
\begin{corollary}\label{nnc}
\label{Levy_SBM}If the characteristic exponent of a unimodal (isotropic) Levy process $X$
satisfies global WLSC and WUSC  (with exponents $0<\la\le \ua<2$),
then there is a complete subordinate Brownian motion
with
comparable
characteristic exponent, L\'{e}vy measure and transition density.
\end{corollary}

\kb{We see from Corollary~\ref{nnc} that under global lower and upper scaling, the asymptotics of the characteristic exponent, L\'{e}vy measure and transition density of complete subordinate Brownian motions are representative for all unimodal L\'evy processes. This is certainly not the case under local scaling, see Section~\ref{sex}.
To indicate an application
 of Corollary~\ref{nnc}, 
we remark that the
boundary Harnack principle with explicit decay rate \cite[Theorem 4.7]{2012arXiv1212.3092K} can now be extended to 
general  unimodal L\'evy processes with continuous L\'evy density and global scaling, see \cite[Proposition~7.6]{BGR2}.
}

We close our discussion with a related result, which negotiates the asymptotics of the L\'evy density (at zero) and the L\'evy-Khintchine exponent (at infinity)
under {\it approximate} unimodality and weak local scaling conditions. We hope the result will help extend the  {\it common bounds}.

Recall our conventions: $0<\la,\ua<2$, $0<\lC\le 1\le \uC<\infty$.
\begin{proposition}\label{tau} Let
$X$ be a pure-jump symmetric
L\'{e}vy process with
L\'{e}vy measure $\nu(dx)=\nu(x)dx$ and characteristic exponent $\psi$.
Suppose that $
\theta\in [0,\infty)$,  constant $c\in (0,1]$ and nondecreasing function $f:(0,\infty)\to (0,\infty)$ are such that
\begin{equation*}\label{au}
c\frac{f(1/|x|)}{|x|^d}\leq \nu(x)\leq
c^{-1}\frac{f(1/|x|)}{|x|^d},\quad 0<|x|<
{1/\theta}.
\end{equation*}
If  $f\in \textrm{WLSC}(\la,
\theta,\lC) \cap \textrm{WUSC}(\ua,
\theta,\uC)$, then {$f(|\xi|)$ and $\psi{(\xi)}$ are comparable for $|\xi|>
\theta$. In fact, there is a complete subordinate Brownian motion whose
characteristic exponent is comparable to $\psi(x)$ for $|\xi|>
\theta$,
and whose L\'{e}vy measure is comparable to $\nu$ on $B_{
1/\theta}\setminus \{0\}$.}
\end{proposition}
\begin{proof}
Let $Y$ be the pure-jump unimodal L\'{e}vy process with L\'{e}vy density
$$\nu^Y(x)=f(1/|x|)|x|^{-d}\textbf{1}_{B_{
1/\theta}}(x), \qquad x\in \Rdz.$$ Let $\psi^Y$ be the characteristic exponent of $Y$.
Using $C=C(d)$ of \eqref{onuzg} and Proposition \ref{Vestimate}, we get
\begin{equation}\label{lbphi}
f(|\xi|) \leq {\pi^2C}\psi^Y(\xi),\qquad {\xi\in \Rdz}.
\end{equation}
On the other hand Corollary \ref{ch1Vpc} yields
${C'}
=C'
(d)$ such that
\begin{equation*}
\psi^Y(\xi)\leq
C'
 \int_{B_{{1/\theta}}}\[(|\xi||z|)^2\wedge1 \]\frac{f(1/|z|)}{|z|^d}dz,\quad {\xi\in \Rd}.
\end{equation*}
By scaling
of $f$, for $|\xi|>\theta
$ we have
\begin{align}
\psi^Y(\xi)&\leq C'
 f(|\xi|)\(\uC|\xi|^{2-\ua}\int_{B_{1/|\xi|}}\frac{dz}{|z|^{d+\ua-2}}+\lC^{-1}|\xi|^{-\la}\int_{
(B_{1/|\xi|})^{{c}}}\frac{dz}{|z|^{d+\la}}\)\nonumber\\
&=C'
\omega_d\(\frac{\uC}{2-\ua}+\frac1{\lC\,\la}\)f(|\xi|).\label{ubphi}
\end{align}
By the symmetry $\nu(z)=\nu(-z)$,
\begin{equation}\label{cppY}
\psi(\xi)=\int_\Rd (1-\cos \langle \xi,z \rangle)\nu(z)dz\approx \psi^Y(\xi)+
\int_{(B_{{1/\theta}})^c} (1-\cos \langle \xi,z \rangle)\nu(z)dz,
\end{equation}
in particular $\psi\approx \psi^Y$ on $\Rd$ if $
\theta=0$. If $
\theta>0$, then the last integral in \eqref{cppY} is bounded by $2\nu\((B_{{1/\theta}})^c\)<\infty$; by Proposition \ref{Vestimate} we have
$$\psi^Y(\xi)\geq \psi^Y(
\theta)/\pi^2>0,
\qquad
\quad |x|
>
\theta,$$
and so $\psi
(\xi)\approx \psi^{{Y}}(\xi)\approx f(|\xi|)$ for $|\xi|
>
\theta$,
where in the second comparison we used  \eqref{lbphi} and \eqref{ubphi}.
It follows that
$
\psi^Y\in  \textrm{WLSC}(\la,
\theta,\lC_1)\cap \textrm{WUSC}(\ua,
\theta,\uC_1)$.

\mr{As in the proof of Theorem \ref{NWSR}, we now construct a complete Bernstein function $\varphi$ such that \eqref{comp10} holds with $\psi^Y$ and $\phi$. Hence $Z$, the complete subordinate Brownian motion determined by $\varphi$, has characteristic exponent $\psi^Z$  comparable  with $\psi^Y$ on $\Rd$.  }
Its  L\'evy density $\nu^Z$ is comparable with
$\nu^Y$ on $B_{r_0/\theta}\setminus\{0\}$ by \eqref{onuzg} and Corollary \ref{nuApprox}, where $r_0=r_0(d,\la,\lC_1,\ua,\uC_1)$.
The comparability of $\nu^Z(x)$
and $\nu^Y(x)$
also takes place on $B_{1/\theta}\setminus B_{r_0/\theta}$ because the functions are bounded from above and below on the set, as follows from \eqref{lmSBM} and monotonicity of $f>0$. Thus, $\nu$ and $\nu^Z$ are comparable on $B_{1/\theta}$.
\end{proof}
To clarify,
the semigroup
of the process $X$ in Proposition~\ref{tau}
is not necessarily unimodal, hence
its estimates by $f$
call for other methods, e.g. those based on
$\nu$, mentioned in Section~\ref{introit}.

\vspace{10pt}

\noindent
{\large {\textbf{Acknowledgements}}}\\
We thank the referee for helpful comments on the presentation and scaling conditions.
Tomasz Grzywny was supported by the Alexander von Humboldt Foundation and
expresses his gratitude for
the hospitality of Technische Unversit\"{a}t Dresden, where the paper was written in main part.
Krzysztof Bogdan gratefully thanks the Department of Statistics of Stanford University for hospitality during his work on the paper.

\end{document}